\documentclass[10pt,a4paper]{article}
\usepackage{amssymb}
\usepackage{amsmath, amsthm, amscd, amsfonts, amssymb, graphicx, color}
\usepackage[bookmarksnumbered, colorlinks, plainpages]{hyperref}
\newtheorem{thm}{Theorem}[section]
\newtheorem{cor}[thm]{Corollary}
\newtheorem{lem}[thm]{Lemma}
\newtheorem{prob}[thm]{Problem}

\newtheorem{obs}[thm]{Observation}
\newtheorem{ques}[thm]{Question}
\newtheorem{prop}[thm]{Proposition}
\newtheorem{defn}[thm]{Definition}

\numberwithin{equation}{section}
\numberwithin{equation}{subsection}

\begin{document}

\title{Total mixed domination in graphs}

\author{$^{1}$Farshad Kazemnejad, $^{2}$Adel P. Kazemi and $^{3}$Somayeh Moradi \\[1em]
$^{1,2}$ Department of Mathematics, \\ University of Mohaghegh Ardabili, \\ P.O.\ Box 5619911367, Ardabil, Iran. \\
$^1$ Email:  kazemnejad.farshad@gmail.com\\
$^2$ Email: adelpkazemi@yahoo.com\\
[1em]
$^3$ Department of Mathematics, \\ School of Sciences, Ilam University, \\
P.O.Box 69315-516, Ilam, Iran. \\
$^3$ Email: somayeh.moradi1@gmail.com \\
}

\maketitle

\begin{abstract}
For a graph $G=(V,E)$, we call a subset $ S\subseteq V \cup E$ a total mixed dominating set of $G$ if each element of $V \cup E$ is either adjacent or incident to an element of $S$, and the total mixed domination number $\gamma_{tm}(G)$ of $G$ is the minimum cardinality of a total mixed dominating set of $G$. In this paper, we initiate to study the total mixed domination number of a connected graph by giving some tight bounds in terms of some parameters such as order and total domination numbers of the graph and its line graph. Then we discuss on the relation between total mixed domination number of a graph and its diameter. Studing of this number in trees is our next work. Also we show that the total mixed domination number of a graph is equale to the total domination number of a graph which is obtained by the graph. Giving the total mixed domination numbers of some special graphs is our last work.
\\[0.2em]

\noindent
Keywords: Total mixed domination, total domination, total graph.
\\[0.2em]

\noindent
MSC(2010): 05C69.
\end{abstract}

\pagestyle{myheadings}
\markboth{\centerline {\scriptsize  F. Kazemnejad, A. P. Kazemi, S. Moradi}}     {\centerline {\scriptsize F. Kazemnejad, A. P. Kazemi, S. Moradi,~~~~~~~~~~~~~~~~~~~~~~~~~~~~~~~~~~~~~~~~~~~~~~~~Total mixed domination in graphs}}


\section{\bf Introduction}

All graphs considered here are non-empty, finite, undirected and simple. For
standard graph theory terminology not given here we refer to \cite{West}. Let $%
G=(V,E) $ be a graph with the \emph{vertex set} $V$ of \emph{order}
$n(G)$ and the \emph{edge set} $E$ of \emph{size} $m(G)$. $N_{G}(v)$ and
$N_{G}[v]$ denote the \emph{open neighborhood} and the \emph{closed neighborhood} of a
vertex $v$, respectively, while $\delta =\delta (G)$ and $\Delta =\Delta (G)$ denote the \emph{minimum} and \emph{maximum degrees} of $G$, respectively. Also we define $N_{E(G)}(v)=\{e\in E(G)~|~v\in e\}$ for any vertex $v$, $N_{G}(e)=\{v\in V(G)~|~v\in e\}$ and $N_{E(G)}(e)=\{e'\in E(G)~|~e \mbox{ and } e' \mbox{ are adjacent} \}$ for any edge $e$, and $N_{T(G)}(x)=N_{G}(x) \cup N_{E(G)}(x)$ for any element $x\in V(G)\cup E(G)$. For two vertices $u$ and $v$ in a connected graph $G$ the \emph{distance} between $u$ and $v$ is the minimum length of a shortest $(u,v)-$path in $G$ and is denoted by $d(u,v)$. The maximum distance among all pairs of vertices of $G$ is the \emph{diameter} of $G$, which is denoted by \emph{diam(G)}. A \emph{Hamiltonian path} in a graph $G$ is a path which contains every vertex of $G$.

\vskip 0.2 true cm

We write $K_{n}$, $C_{n}$ and $P_{n}$ for a \emph{complete graph}, a \emph{cycle} and a \emph{path} of order $n$, respectively, while $G[S]$, $W_n$ and $K_{n_1,n_2,\ldots,n_p}$ denote the subgraph of $G$ \emph{induced} by a subset $S\subseteq V(G)\cup E(G)$ of $G$, a \emph{wheel} of order $n+1$, and a \emph{complete $p$-partite graph}, respectively. The \emph{complement} of a graph $G$, denoted by $\overline{G}$, is a graph with the vertex set $V(G)$ and for every two vertices $v$ and $w$, $vw\in E(\overline{G})$ if and only if $vw\not\in E(G)$. The \emph{line graph} $L(G)$ of $G$ is a graph with the vertex set $E(G)$ and two vertices of $L(G)$ are adjacent when they are incident in $G$.

\vskip 0.2 true cm

Domination in graphs is now well studied in graph theory and the literature
on this subject has been surveyed and detailed in the two books by Haynes,
Hedetniemi, and Slater~\cite{hhs1, hhs2}. A famous type of them is total domination. The literature on the subject on total domination in graphs has been surveyed and detailed in the recent book~\cite{HeYe13} by Henning and Yeo.

\begin{defn}
\emph{A subset $S\subset V$ of a graph $G$ is a} total dominating set, \emph{briefly TDS, of $G$ if each vertex of $V$ is adjacent to a vertex in $S$, and the} total domination number $\gamma_t(G)$ \emph{of $G$ is the minimum cardinality of a total dominating set.}
\end{defn}

Y. Zhao, L. Kang, and M. Y. Sohn in \cite{Y. Zhao}  presented another domination number as follows.

\begin{defn}
\emph{\cite{Y. Zhao} A subset $S\subseteq V\cup E$ of a graph $G$ is a} mixed dominating set, \emph{briefly MDS, of $G$ if each element of $(V\cup E)-S$ is either adjacent or incident to an element of $S$, and the} mixed domination number $\gamma_{m}(G)$ \emph{of $G$ is the minimum cardinality of a mixed dominating set.}
\end{defn}

Here, we initiate studying of total mixed domination in graphs that is a generalization of mixed domination by adding the concept of total in the following meaning.

\begin{defn}
\emph{A subset $S\subseteq V\cup E$ of a graph $G$ with $\delta(G)\geq 1$ is a} total mixed dominating set, \emph{briefly TMDS, of $G$ if each element of $V \cup E$ is either adjacent or incident to an element of $S$, and the} total mixed domination number $\gamma_{tm}(G)$ \emph{of $G$ is the minimum cardinality of a total mixed dominating set.}
\end{defn}


\vskip 0.2 true cm

The goal of this paper is to initiate studying of total mixed domination number of a graph. First in section 2, we give some tight lower and upper bounds for the total mixed domination number of a connected graph in terms of some parameters such as the order of the graph or the total domination numbers of the graph and its line graph. Also we discuss on the relation between the total mixed domination number of a graph with its diameter. Studing of total mixed domination number of trees is our next work. Also, we show that the total mixed domination number of a graph is equale to the total domination number of a graph which is obtained by the graph, named total graph. Finally in last section, we will calculate the total mixed domination number of special classes of graphs including paths, cycles, complete bipartite graphs, complete graphs and wheels. 

\vskip 0.2 true cm
Here, we fix a notation for the vertex set, the edge set and open neighbrhood of a graph which are used thorough this paper. For a graph $G$ with the vertex set $V=\{v_i|\ 1\leq i\leq n\}$, $\mathbb{E}(G)$ or simply $\mathbb{E}$ denotes the edge set of $G$ in which an edge $v_{i}v_{j}$ is denoted by $e_{ij}$. Then $V(L(G))=\mathbb{E}$, and the edge set of $L(G)$ is the set $\{e_{ij}e_{ik}~|~ e_{ij}, e_{ik}\in \mathbb{E}\}$.  A min-TDS/ min-TMDS of $G$ denotes a TDS/ TMDS of $G$ with minimum cardinality. Also we agree that \emph{a vertex $v$ dominates an edge} $e$ or \emph{an edge $e$ dominates a vertex} $v$ mean $v\in e$. Similarly, we agree that \emph{an edge dominates another edge} means they have a common vertex.

\section{\bf Main results}

\subsection{\bf Some general bounds}
Here, we give some tight bounds for the total mixed domination number of a connected graph in terms of some parameters such as order of the graph or the total domination numbers of the graph and its line graph. Also we discuss on the relation between the total mixed domination number of  a graph  and its diameter. First an observation.
\begin{obs}
\label{SG,SL}
Let $G$ be a graph with the vertex set $V=\{v_i~|~1\leq i\leq n\}$ and $\delta(G)\geq 1$.
\\$\bullet$ A subset $S\subseteq \mathbb{E}$ is a TDS of $L(G)$ if and only if $\{i,j~|~e_{ij}\in \mathbb{E} \}\cap \{i,j~|~e_{ij}\in S\}\neq \emptyset$.\\
$\bullet$ A TDS $S$ of $L(G)$ is a TMDS of $G$ if and only if $\{i,j~|~e_{ij}\in S\}=\{1,2,\cdots, n\}$.\\
$\bullet$ A TDS $S$ of $G$ is a TMDS of $G$ if and only if $\overline{S}$ is independent in $G$.
\end{obs}

\begin{thm}
\label{max{gama_t  G, gama_t  L(G)} =< gama_t T(G)=<gama_t  G+gama_t  L(G)}
Let $G$ be a connected graph with $\delta(G)\geq 1$. Then
 \[
\max\{\gamma_t  (G), \gamma_t  (L(G))\}  \leq \gamma_{tm}(G) \leq \gamma_t  (L(G)) +\gamma_t  (G),
\]
and the bounds are tight.
\end{thm}

\begin{proof}
Let $G$ be a connected graph with the vertex set $V=\{v_{1},\ldots,v_{n}\}$ and edge set $\mathbb{E}$ in which $e_{ij}$ denotes edge $v_{i}v_{j}$. Then $V(L(G))=\mathbb{E}$ and $E(L(G))=\{e_{ij} e_{i^{'}j^{'}}~|~\{i,j\} \cap \{i^{'},j^{'}\} \neq \emptyset\}$. Since the union of a TDS of $G$ and a TDS of $L(G)$ is a TMDS of $G$, we have $\gamma_{tm}(G) \leq \gamma_t  (L(G)) +\gamma_t  (G)$. To prove the lower bound, let $S$ be a min-TMDS of $G$. If either every vertex of $V$ is dominated by a vertex in $S\cap V$, or every edge of $\mathbb{E}$ is dominated by an edge of $S \cap \mathbb{E}$, then $S\setminus \mathbb{E}$ or $S\setminus V$ is a TDS of $G$ or $L(G)$, respectively, and there is nothing to prove. Otherwise,
\begin{equation*}\label{eqq1}
\begin{array}{lll}
S'_G &= & (S\setminus \mathbb{E})\cup \{v_i~|~v_i \mbox{ is adjacent to some } v_j\in S \mbox{ such that } N_{T(G)}(v_j) \cap S\subseteq \mathbb{E}\}\\
      & \cup & \{v_{i}, v_{j}~|~ e_{ij}, e_{jk} \in S \mbox{ but } v_{j}, v_{k} \notin S\}
\end{array}
\end{equation*}
is a TDS of $G$ with cardinality at most $|S|$. Since also by changing the roles of $G$ and $L(G)$ we may obtain a TDS $S'_L$ of $L(G)$  with cardinality at most $|S|$, we have proved 
\[
\max \{ \gamma_t  (G), \gamma_t  (L(G))\}  \leq \max \{|S'_{G}|,|S'_L|\} \leq |S|=\gamma_{tm}(G).
\]

The lower bound is tight for the complete graphs $K_{3n}$ by Propositions \ref{gamma_t T(K_n)} and \ref{gama_t(L(K_n))=2n/3} when $\max \{ \gamma_t  (G), \gamma_t  (L(G))\}=\gamma_t  (L(G))$. The case $\max \{ \gamma_t  (G), \gamma_t  (L(G))\}=\gamma_t  (G)=\gamma_{tm} (G)$ is discussed in Corollary \ref{gama_t  G = gama_t T(G)}. To show that the upper bound is tight, consider the graph $G$ illustrated in Figure \ref{fi:proofexample1} with $\gamma_t  (G)=2$ (because $\{v_1,v_5\}$ is a min-TDS) and $\gamma_t  (L(G))=4$ by Observation \ref{SG,SL} (because $\{e_{12},e_{23},e_{56},e_{67}\}$ is a TDS of $L(G)$ and for any set $\{e_{ij},e_{jk},e_{k\ell}\}$ we have $\{i,j,k,\ell\}\neq \{0,1,\cdots 9\}$). So it is sufficient to prove $\gamma_{tm}(G)=6$. First since $\{v_1,v_2,v_3,v_5,v_6,v_7\}$ is a TMDS of $G$, we have $\gamma_{tm}(G)\leq 6$. Let $G_1$ and $G_2$ be the subgraphs of $G$ induced by $\{v_i~|~ 0\leq i\leq 4\}$ and $\{v_i~|~ 5\leq i\leq 9\}$, respectively, which are isomorphic together (see Figure \ref{fi:proofexample2}). And let also $S$ be a TMDS of $G$ such that $|S\cap (V(G_2) \cup \mathbb{E}(G_{2})) |\geq |S\cap (V(G_1) \cup \mathbb{E}(G_{1}))|\geq 2$. By the contrary, let $|S\cap (V(G_1) \cup \mathbb{E}(G_{1}))|= 2$. Since $N_{T(G)}(v_0)\cap S\subseteq \{v_1,e_{01}\}$, we have $S\cap (V(G_1) \cup \mathbb{E}(G_{1}))=\{e_{01},w\}$ for some $w\in \{v_1,v_0,e_{1j}~|~ 2\leq j\leq 4\}$ or $S\cap (V(G_1) \cup \mathbb{E}(G_{1}))=\{v_1,w\}$ for some $w\in \{v_j ,e_{1j}~|~ 2\leq j\leq 4\}\cup\{v_0,e_{01}\}$. So $S\cap (V(G_1) \cup \mathbb{E}(G_{1}))$ is one the sets $\{e_{01},v_j\}$ for some $j=0,1$, or  $\{e_{01},e_{1j}\}$ for some $j=2,3,4$, or $\{v_1,v_{j}\}$ for some $j=0,2,3,4$, or $\{v_1,e_{1j}\}$ for some $j=0,2,3,4$. Since in each case $N_{T(G)}(e_{k\ell})\cap S=\emptyset$ for some $k,\ell \in \{2,3,4\}-\{j\}$, we conclude $|S\cap (V(G_1) \cup \mathbb{E}(G_{1}))|\geq 3$ and so  $\gamma_{tm}(G)= 6$.
\end{proof}

\begin{figure}[ht]
\centerline{\includegraphics[width=12.8cm, height=4.5cm]{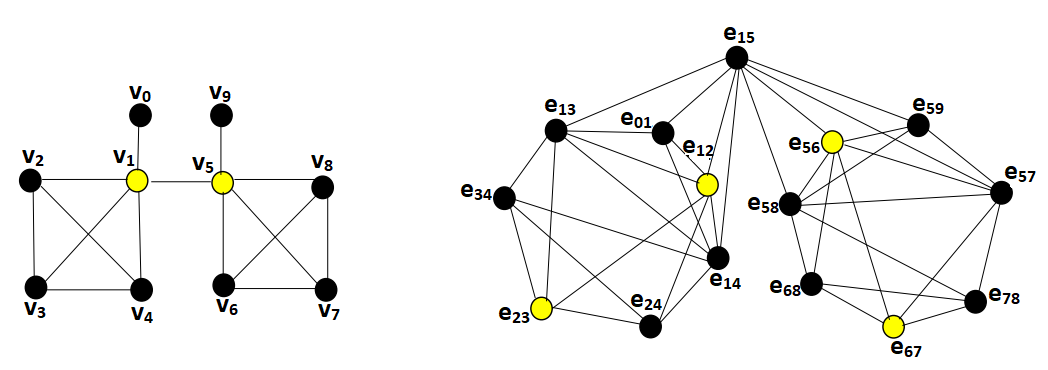}}
\vspace*{-0.3cm}
\caption{ The illustration of $G$ (left) and $L(G)$ (right).}\label{fi:proofexample1}
\end{figure}
\begin{figure}[ht]
\centerline{\includegraphics[width=6.3cm, height=4cm]{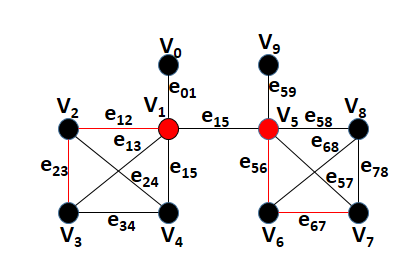}}
\vspace*{-0.3cm}
\caption{A min-TMDS of the graph $G$.}\label{fi:proofexample2}
\end{figure}


Corollary \ref{gama_t  G = gama_t T(G)}, which is obtained by Observation \ref{SG,SL}, shows that the lower bound in Theorem \ref{max{gama_t  G, gama_t  L(G)} =< gama_t T(G)=<gama_t  G+gama_t  L(G)} is tight for the case $\max \{ \gamma_t  (G), \gamma_t  (L(G))\}=\gamma_t  (G)=\gamma_{tm} (G)$. For example, for any complete bipartite graph $G=K_{1,n}$ and any double star graph $G=S_{1,n,n}$, $ \gamma_t(G) =\gamma_{tm}(G)$ (recall that a \emph{double star graph} $S_{1,n,n}$ is obtained from the complete bipartite graph $K_{1,n}$ by replacing every edge  by a path of length 2). For an example, the set of yellow points $\{v_{0}, v_{1}, v_{2}, v_{3}\}$ in Figure \ref{fi:proofexample3} is a min-TDS and a min-TMDS of $S_{1,3,3}$.
\begin{figure}[ht]
\centerline{\includegraphics[width=4cm, height=3cm]{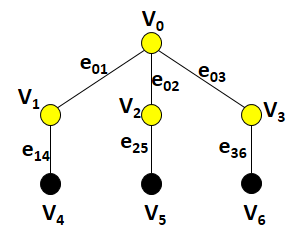}}
\vspace*{-0.3cm}
\caption{A min-TDS and a min-TMDS of the double star $S_{1,3,3}$. \label{fi:proofexample3}}
\end{figure}
\begin{cor}
\label{gama_t  G = gama_t T(G)}
For any graph $G$ which has a min-TDS such that its complement is an independent set, $ \gamma_t(G) =\gamma_{tm}(G)$.
\end{cor}
As some research problems, naturally the next problems can be arised.
\begin{prob}
1. For any graph $G$, is it true that $\gamma_{tm}(G)= \gamma_t(G) $ if and only if it has a min-TDS such that its complement is an independent set of $G$?

2. Find some families of non-complete graphs $G$ with $\gamma_{tm}(G)= \gamma_t(L(G)) $.

3. Find some families of connected graphs $G$ satisfy $\gamma_{tm}(G)= \gamma_t(L(G)) +\gamma_t(G)$.
\end{prob}


The next theorem improves the upper bound given in Theorem \ref{max{gama_t  G, gama_t  L(G)} =< gama_t T(G)=<gama_t  G+gama_t  L(G)} when either the line graph $L(G)$ has a min-TDS $D$ such that there exist less than $\gamma_t(G)$ disjoint maximal cliques in $L(G)\setminus D$, or $G$ has a min-TDS $S$ such that the minimum size of vertex cover of $G\setminus S$ is less than $\gamma_t(L(G))$ (recall that a \emph{vertex cover} of $G$ is a subset $S$ of $V(G)$ such that each edge of $G$ has a vertex in $S$ and the $\beta(G)$ denotes the minimum size of a vertex cover of $G$).
\begin{thm}
\label{gama_t L(G) = gama_t T(G)+k}
For any connected graph $G$ with $\delta(G)\geq 1$, let $c_{D}$ be the minimum number of disjoint maximal cliques in $L(G)\setminus D$ where $D$ is a min-TDS of $L(G)$, and let $\beta(G\setminus S)$ be the minimum size of a vertex cover of $G\setminus S$ where $S$ is a min-TDS of $G$. Then, by the assumptions $c_{L(G)}=\min\{c_D~|~ D \mbox { is a min-TDS of } L(G)\}$ and $\beta_G=\min\{ \beta(G\setminus S)~|~ S \mbox{ is a min-TDS of } G\}$,
\[
\gamma_{tm} (G) \leq \min\{\gamma_t(L(G))+c_{L(G)},\gamma_t(G)+\beta_G\},
\]
and this bound is tight.
\end{thm}
{\begin{proof}
We show that the total mixed domination number of $G$ is at most the minimum of the given set, when $G=(V,\mathbb{E})$ is a  connected graph with $\delta(G)\geq 1$ and $V=\{v_{1},\ldots,v_{n}\}$, and $e_{ij}$ denotes edge $v_{i}v_{j}$. So  $E(L(G))=\{e_{ij}e_{ik}~|~ e_{ij}, e_{ik}\in \mathbb{E}\}$. First we prove $\gamma_{tm} (G) \leq \gamma_t(L(G))+c_{L(G)}$. Let $D$ be a min-TDS of $L(G)$ such that $c_{L(G)}=c_{D}$. Obviously $D$ dominates all elements of $E(L(G)) \cup \{v_{i} ~|~ e_{ij} \in D \mbox{ for some } j\}$. Let $\mathcal{C}$ be a subset of $E(L(G))$ with cardinality $c_{L(G)}$ such that every maximal clique of $L(G)\setminus D$ has exactly one vertex in $\mathcal{C}$, and let $v_i$ be a vertex that does not dominated by $D$. Since $N_{L(G)}(v_i)=\{e_{ij}|\ v_iv_j\in E(G)\}$ is a maximal clique in $L(G)\setminus D$, we are sure that $v_i$ is dominated by the unique vertex of $\mathcal{C}\cap N_{L(G)}(v_i)$. Thus $D \cup \mathcal{C}$ is a TMDS of $G$, and so $\gamma_{tm} (G) \leq |D \cup \mathcal{C}|=\gamma_t(L(G))+c_{L(G)}$. In a similar way, the inequality $\gamma_{tm} (G) \leq \gamma_t(G)+\beta_G$ can be proved and  this completes our proof. 
\vspace{0.2cm}

As we show in the next lemma, this upper bound is tight for any wheel of order at least 4. 
 \end{proof}

As we show in below, our motivation to sate Theorem \ref{gama_t L(G) = gama_t T(G)+k} is the existance of graphs that the upper bound in Theorem \ref{gama_t L(G) = gama_t T(G)+k} is better than the upper bound in Theorem \ref{max{gama_t  G, gama_t  L(G)} =< gama_t T(G)=<gama_t  G+gama_t  L(G)} for them. Let $W_n$ be a wheel of order $n+1\geq 4$ with the vertex set $ V=\{v_{i}~|~ 0 \leq i \leq n \} $ and the edge set $ \mathbb{E}=\{ e_{0i}, e_{i(i+1)} ~|~ \mbox{for}~ 1\leq i \leq n \}$. Then, since $S=\{v_{0}, v_{1}\}$ is a min-TDS of $W_n$, $\gamma_{t}(W_n)=2$. On the other hand, $W_n \setminus S \cong P_{n-1}$ implies $\beta_{W_n}=\beta(P_{n-1})=\lfloor (n-1)/2 \rfloor$. Hence $\gamma_t(W_n)+\beta_{W_n}=2+\lfloor (n-1)/2 \rfloor=\lceil n/2 \rceil+1$. Since $\gamma_t(L(W_n))=\lceil n/2 \rceil$ by Lemma \ref{gamma_t(L(W_n))=lceil n/2 rceil }, we have
\begin{equation*}
\begin{array}{lll}
\min\{\gamma_t(L(W_n))+c_{L(W_n)},\gamma_t(W_n)+\beta_{W_n}\}& = &  \min \{\lceil n/2 \rceil+c_{L(W_n)}, \lceil n/2 \rceil+1\}\\
     &  = &  \lceil n/2 \rceil+1 \\
     & =  & \gamma_{tm}(W_n) ~~~\mbox{ (by Proposition \ref{gamma_t(T(W_n))=lceil n/2 rceil +1}) }\\
     & < &  \gamma_t(L(W_n))+\gamma_t(W_n).
     \end{array}
\end{equation*} 
\begin{lem}
\label{gamma_t(L(W_n))=lceil n/2 rceil }
 For any wheel $W_n$ of order $n+1\geq 4$, $\gamma_{t}(L(W_n))=\lceil n/2 \rceil$.
\end{lem}

\begin{proof}
Let $W_n$ be a wheel of order $n+1\geq 4$ with the vertex set $ V=\{v_{i}~|~ 0 \leq i \leq n \} $ and the edge set $ \mathbb{E}=\{ e_{0i}, e_{i(i+1)} ~|~ \mbox{for}~ 1\leq i \leq n \}$ (note: $n+1$ is considered 1 to modulo $n$). Let $S$ be a TDS of $L(W_n)$ where $V(L(W_n))=\mathbb{E}$. Then $\{i,i+1\} \cap \{1,2,\ldots,n\} \neq \emptyset$  for each $1\leq i \leq n$ because of $N_{W_n}(e_{i(i+1)}) \cap S=\{v_i,v_{i+1}\} \cap S \neq \emptyset$. Hence $|S| \geq \lceil n/2 \rceil$. Now since $\{e_{0(2i-1)}~|~1\leq i \leq \lceil n/2 \rceil\}$ is a TDS of $L(W_n)$, we have $\gamma_{t}(L(W_n))=\lceil n/2 \rceil $.
\end{proof}

We know for any graph $G$ with a non-empty edge set, $\gamma_{tm}(G) \geq 2$. Corollary \ref{gama_t  G = gama_t T(G)} charactrises graphs $G$ satisfy $\gamma_{tm}(G)=2$. The next theorem gives a sufficient condition for that the total mixed domination number of a graph be at least 3.

\begin{thm}
\label{gama_t(T(G))=2--> diam(G) =<3}
For any connected graph $G$ of order at least $2$, $\gamma_{tm}(G)=2$ implies $diam(G) \leq 3$.
\end{thm}

\begin{proof}
Let $G=(V,\mathbb{E})$ be a connected graph with $diam(G)\geq 4$ and  $\delta(G)\geq 1$ in which $V=\{v_{1},\ldots,v_{n}\}$. The condition $diam(G)\geq 4$ implies that $G$ has an induced path $P$ of length at least 4. Let $S$ be a min-TMDS of $G$  of cardinality 2. If $S$ is $\{v_{i}, v_{j}\}$ or $\{v_{i}, e_{ij}\}$ for some $i,j$, then $N_{T(G)}(e_{pq}) \cap S= \emptyset$ for some $v_p,v_q\in V(P)$, a contradiction. Also if $S=\{e_{ij}, e_{jk}\}$  for some $i,j,k$, then $N_{T(G)}(v_{\ell}) \cap S= \emptyset$ for some $v_{\ell}\in V(P)$ such that $\ell\neq i,j,k$, a contradiction. So $\gamma_{tm}(G)\neq 2$.
\end{proof}
Now, we present another upper bound for $\gamma_{tm}(G)$ in term of the order of the graph which is tight by Proposition \ref{gamma_t T(K_n)}.
\begin{thm}
\label{gama_t T(G) =< 2n/3}
For any connected graph $G$ of order $n \geq 2$ which has a Hamiltonian path,
\begin{equation*}
\gamma_{tm}(G)\leq \left\{
\begin{array}{ll}
\lfloor 2n/3\rfloor  & \mbox{if }n \equiv 0 \pmod{3}, \\
\lfloor 2n/3\rfloor +1 & \mbox{if }n \equiv 1,2 \pmod{3}.
\end{array}
\right.
\end{equation*}
\end{thm}

\begin{proof}
Let $P:v_{1}v_{2}\cdots v_{n}$ be a Hamiltonian path in $G$. Since each of the sets
\begin{equation*}
\begin{array}{ll}
S_0=\{e_{(3i+1)(3i+2)}, e_{(3i+2)(3i+3)},~|~ 0 \leq i \leq \lfloor n/3 \rfloor -1\} & \mbox{if }n \equiv 0 \pmod{3}, \\
S_1=S_0\cup\{e_{(n-1)n}\}  & \mbox{if }n \equiv 1 \pmod{3}, \\
S_2=S_0\cup\{e_{(n-2)(n-1)}, e_{(n-1)n}\} & \mbox{if }n \equiv 2 \pmod{3},
\end{array}
\end{equation*}
is a TMDS of $G$, the result holds.
\end{proof}

It can be easily verified that Theorem \ref{gama_t T(G) =< 2n/3} is true for any connected graph of order at most 5. So the existance of a Hamiltonian path in a graph is not a necessary condition in the theorem, and naturally the following question arises.
\begin{ques}
\label{gamma_t(T(G)) leq lfloor2n/3 rfloor+1}
 Is Theorem \ref{gama_t T(G) =< 2n/3} true for another family of graphs?
\end{ques}
\subsection{Trees}

The facts that $diam(C_4)=2$ and $\gamma_{tm}(C_4)=3$ show that the converse of Theorem \ref{gama_t(T(G))=2--> diam(G) =<3} is not true in general. But next theorem shows that it holds for trees.

\begin{thm}
\label{gama_t(T(G))=2 <--> diam(G) =<3}
For any tree $\mathbb{T}$ of order at least $2$, $\gamma_{tm}(\mathbb{T})=2$ if and only if $diam(\mathbb{T}) \leq 3$.
\end{thm}

\begin{proof}
By Theorem  \ref{gama_t(T(G))=2--> diam(G) =<3}, it is sufficient to prove that $diam(\mathbb{T}) \leq 3$ implies $\gamma_{tm}(\mathbb{T})=2$. If $diam(\mathbb{T})=1$, then $\mathbb{T} \cong K_{2}$  and so $\gamma_{tm}(\mathbb{T})=2$. If $diam(\mathbb{T})=2$, then $\mathbb{T}$ is isomorphic to the complete bipartite graph $K_{1,n-1}$ and so $\gamma_{tm}(\mathbb{T})=2$ by Proposition \ref{gamma_t  (T(K_{m,n}))}. Now let $diam(\mathbb{T}) =3$. Then $\mathbb{T}$ is a tree which is obtained by joining the central vertex $v$ of a tree $K_{1,p}$ and the central vertex $w$ of a tree $K_{1,q}$ where $p+q=n-2$. Since $\{v,w\}$ is a TMDS of $\mathbb{T}$, we have $\gamma_{tm}(\mathbb{T})=2$.
\end{proof}

Next theorem improves the upper bound given in Theorem \ref{gama_t T(G) =< 2n/3} for trees.
\begin{thm}
\label{gamma_{tm}(mathbb{T}) leq  2n/3}
For any tree $\mathbb{T}$ of order $n \geq 3$, $\gamma_{tm}(\mathbb{T}) \leq \lfloor 2n/3 \rfloor$.
\end{thm}

\begin{proof}
Let $\mathbb{T}=(V,\mathbb{E})$ be a tree in which $V=\{v_{i}~|~1\leq i \leq n\}$.  Choose a leaf $v$ of $\mathbb{T}$ and label each vertex of $\mathbb{T}$ with its distance from $v$ to modolu $3$. This partitions $V$ to the three independent sets $A_0$, $A_1$ and $A_2$ where $A_i=\{u\in V~|~d_{\mathbb{T}}(u,v) \equiv i \pmod{3}\}$ for $0\leq i \leq 2$. Then by the piegonhole principle at least one of them, say $A_0$, contains at least one third of the vertices of $\mathbb{T}$, and so $|A_1\cup A_2|\leq \lfloor 2n/3 \rfloor$. We see that every \emph{internal vertex}, which is a vertex of degree at least two, and every  leaf $v_i \in V(\mathbb{T})-A_1\cup A_2$ is adjacent to some vertex in  $A_1\cup A_2$. If needed, we replace every leaf $v_i\in A_1\cup A_2$ by an its neighbour out of $A_1\cup A_2$. The obtained set $S$ by this way is a TMDS of $\mathbb{T}$. Because obviously $N_{\mathbb{T}}(v_{i}) \cap S \neq \emptyset$ for each $v_i \in V(\mathbb{T})$, and $\{v_{i},v_{j}\} \cap S \neq \emptyset$ for each $e_{ij} \in \mathbb{E}$ (because $d_{\mathbb{T}}(v,v_{i}) \not\equiv d_{\mathbb{T}}(v,v_{j}) \pmod{3}$), and so every $e_{ij}\in \mathbb{E}$ is dominated by $v_i\in S$ or $v_j\in S$. Therefore $\gamma_{tm}(\mathbb{T}) \leq |S|\leq \lfloor 2n/3 \rfloor$.
\end{proof}

By Proposition \ref{gamma_t  (T(H o P_{m}))} the upper bound $\lfloor 2n/3 \rfloor$ in Theorem \ref{gamma_{tm}(mathbb{T}) leq  2n/3} is tight for any 2-corona $\mathbb{T} \circ P_{2}$ in which $\mathbb{T}$ is a tree of order $n \geq 3$. We recall that the 2-\emph{corona} $G \circ P_{2}$ of a graph $G$ is the graph obtained from $G$ by adding a path of order 2 to each vertex of $G$.

\begin{prop}
\label{gamma_t  (T(H o P_{m}))}
For any connected graph $G$ of order $n\geq 2$, $ \gamma_{tm}(G \circ P_2)=2n$.
\end{prop}

\begin{proof}
Let $G=(V,\mathbb{E})$ be a connected graph in which $V=\{v_i~|~1\leq i \leq n\}$. Then $V(G \circ P_{2})=\{v_i~|~1\leq i \leq 3n\}$ and $E(G \circ P_{2})= \mathbb{E} \cup \{e_{i(n+i)}, e_{(n+i)(2n+i)}~|~1\leq i \leq n\}$.  Since $\{v_i, v_{n+i}~|~1\leq i \leq n\}$ is a TMDS of $G \circ P_{2}$, we have $ \gamma_{tm}(G \circ P_{2})\leq 2n$. 
\vspace{0.2cm}

Now let $S$ be a min-TMDS of $G \circ P_{2}$. Then $\{v_{n+i}, e_{(n+i)(2n+i)}\}\cap S$ contains an element $w_i$ (because $N_{T(G \circ P_2)}(v_{2n+i})\cap S\neq \emptyset$) for each $1\leq i \leq n$. Since also every $w_i$ must be dominated by an element $w'_i\in N_{T(G \circ P_2)}(w_{i})\cap S$, and all of the elements $w_i$ and $w'_i$ are distinct, we conclude that $S$ includes the set $\{w_i, w'_i~|~1\leq i \leq n\}$ of cardinality $2n$, and so $ \gamma_{tm}(G \circ P_{2}) \geq 2n$, which completes our proof. 

\vskip 0.2 true cm

The set of 
yellow points  $\{v_i~|~ 1\leq i \leq 12\}$ in Figure \ref{fi:proofexample4} shows a min-TMDS of $P_{6} \circ P_{2}$.
\begin{figure}[ht]
\centerline{\includegraphics[width=8cm, height=3.5cm]{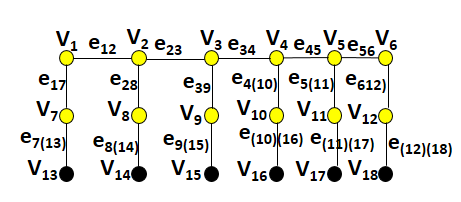}}
\vspace*{-0.3cm}
\caption{A min-TMDS of $P_{6} \circ P_{2}$.}\label{fi:proofexample4}
\end{figure}

\end{proof}

\subsection{\bf Total graphs}

Behzad in \cite{Behzad} defined total of a graph as following:
\begin{defn}
\label{TotalGraph}
\emph{\cite{Behzad} The} total graph $T(G)$ \emph{of a graph $G=(V,\mathbb{E})$ is the graph whose vertex set is $V\cup \mathbb{E}$ and two vertices are adjacent whenever they are either adjacent or incident in $G$.}
\end{defn}


It is obvious that if $G$ has order $n$ and size $m$, then $T(G)$ has order $n+m$ and size $3m+|E(L(G))|$, and also $T(G)$ contains both $G$ and $L(G)$ as two induced subgraphs and it is the largest graph formed by adjacent and incidence relation between graph elements. In Figure \ref{fi:proofexample10} see the total graph of graph $G$ given in Figure  \ref{fi:proofexample1}, for an example.

\vskip 0.2 true cm

It is clear that a total mixed dominating set of a graph $G$ corresponds with a total dominating set of total graph $T(G)$ of $G$. Hence we have the next theorem, and so to find the total mixed domination number of a graph we may calculate the total domination number of total of the graph.

\begin{thm}
\label{gama_{tm}(G)=gama_t(T(G)}
For any graph $G$ with $\delta(G)\geq 1$, $\gamma_{tm}(G)=\gamma_t(T(G))$.
\end{thm}

The set of yellow points $\{v_1,v_5,e_{12},e_{23},e_{56}, e_{67}\}$ in Figure  \ref{fi:proofexample10} shows a min-TDS of $T(G)$.
\begin{figure}[ht]
\centerline{\includegraphics[width=7cm, height=4.5cm]{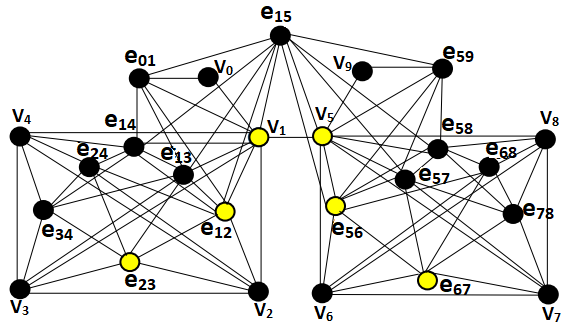}}
\vspace*{-0.3cm}
\caption{A min-TDS of $T(G)$.}\label{fi:proofexample10}
\end{figure}

\section{\bf Special classes of graphs}\label{t80}

In this section, we present formulas for the total mixed domination number of some special classes of graphs. The first two theorems are devoted to paths and cycles.

\begin{prop}
\label{gamma_t T(P_n)}
For any path $P_n$ of order $n\geq 2$,
\begin{equation*}
\gamma_{tm}(P_n)=\left\{
\begin{array}{ll}
4\lceil n/7\rceil -3 & \mbox{if }n\equiv 1 \pmod{7},\\
4\lceil n/7\rceil -2 & \mbox{if }n\equiv 2,3,4 \pmod{7},\\
4\lceil n/7\rceil -1 & \mbox{if }n\equiv 5 \pmod{7},\\
4\lceil n/7\rceil  & \mbox{if }n\equiv 0,6 \pmod{7}.
\end{array}
\right.
\end{equation*}
\end{prop}
\begin{proof}
By Theorem \ref{gama_{tm}(G)=gama_t(T(G)}, we calculate the total domination number of $T(P_n)$ when $P_{n}=(V,\mathbb{E})$ is a path of order $n\geq 2$ in which $V=\{v_{1},v_{2},\cdots, v_{n}\}$  and $\mathbb{E}=\{e_{i(i+1)}~|~1\leq i \leq n-1\}$. Then
$V(T(P_n))=V\cup \mathbb{E}$ and $E(T(P_n)) = \mathbb{E} \cup E(L(P_n))\cup \{e_{i(i+1)}v_i,e_{i(i+1)}v_{i+1}~|~ 1\leq i \leq n-1 \}$ in which $E(L(P_n))=\{e_{i(i+1)}e_{(i+1)(i+2)}~|~ 1\leq i \leq n-2\}$. 
\vskip 0.2 true cm
\textbf{Claim:} There exists a min-TDS $S$ of $T(P_{n})$ with the properties:\\
\textsf{P.1 :} $V(G_i)\subseteq V$ if and only if $V(G_{i+1})\subseteq \mathbb{E}$ for each $i$.\\
\textsf{P.2 :} $|V(G_i)|=2$ for each $i$, perhaps except for $i=w$.\\
\textsf{P.3 :} $V(G_{1})\subseteq V$,\\
in which $G_{1}$, $\cdots$, $G_{w}$ are all connected components of the induced subgraph $T(P_n)[S]$ that appear from the left to the right in $T(P_{n})$.
\vskip 0.2 true cm
By proving the claim, each of the sets 
\begin{equation*}
\begin{array}{ll}
S_0=\{v_{7i+2}, v_{7i+3}, e_{(7i+5)(7i+6)}, e_{(7i+6)(7i+7)}~|~ 0 \leq i \leq \lfloor n/7\rfloor-1\} & \mbox{if }n \equiv 0 \pmod{7}, \\
S=S_0 \cup\{e_{(n-1)n}\}& \mbox{if }n \equiv 1 \pmod{7}, \\
S=S_0 \cup\{v_{n-1},v_{n}\} & \mbox{if }n \equiv 2,3 \pmod{7}, \\
S=S_0 \cup\{v_{n-2},v_{n-1}\} & \mbox{if }n \equiv 4 \pmod{7}\\
S=S_0 \cup\{v_{n-3},v_{n-2},v_{n-1}\} & \mbox{if }n \equiv 5 \pmod{7},\\
S=S_0 \cup\{v_{n-4},v_{n-3},e_{(n-2)(n-1)},e_{(n-1)n}\} & \mbox{if }n \equiv 6 \pmod{7}
\end{array}
\end{equation*}
will be a min-TDS of $T(P_n)$, and this completes our proof.
\vskip 0.2 true cm
While the set $\{v_2,v_3,e_{56},e_{67},v_9,v_{10},v_{11}\}$ in Figure  \ref{fi:proofexample5} shows a min-TDS of $T(P_{12})$  (of red points), it shows a min-TMDS of $P_{12}$  in Figure  \ref{fi:proofexample5.5}.  
\vskip 0.2 true cm
\textbf{Proof of the claim:} Let $S$ be a min-TDS of $T(P_{n})$. We may assume for every $e \in E(T(P_{n})[S])$, $e=v_{i}v_{i+1}$ or $e=e_{i(i+1)}e_{(i+1)(i+2)}$ for some $i$. Because otherwise if $e=v_{i}e_{(i-1)i}$ or $e=v_{i}e_{i(i+1)}$ for some $i$, then we can replace $S$ by $(S-\{v_{i}\})\cup \{e_{i(i+1)}\}$ or $(S-\{e_{i(i+1)}\})\cup \{v_{i+1}\}$, respectively, that each of them is again a min-TDS of $T(P_{n})$. So we may assume that every connected component of $T(P_n)[S]$ is a path of order at least $2$ whose vertex set is either a subset of $V$ or a subset of $\mathbb{E}$. Let $G_{1}$, $\cdots$, $G_{w}$ be all connected components of $T(P_{n})[S]$ that appear from left to right in $T(P_{n})$. By the minimality of $S$, we have $|V(G_i)|\leq 4$ for each $i$. Our proof will be completed by showing that $S$ satisfies the above three property.
\vskip 0.2 true cm
\textsf{P.1:} Let $V(G_j)=\{v_i|\ \ell\leq i\leq k \}$  and $V(G_{j+1})=\{v_i|\ k+2\leq i\leq k+r \}$ for some $j$, $\ell$, $k$, $r$. Then we can replace $S$ by $(S\setminus V(G_{j+1}))\cup \{e_{i(i+1)}|\  k+2\leq i \leq k+r\}$ which is again a min-TDS of $T(P_{n})$. There is a similar proof when both of $V(G_i)$ and $V(G_{i+1})$ are subsets of $\mathbb{E}$.
\vskip 0.2 true cm
\textsf{P.2:} We may consider $V(G_i) \subseteq V$, because the case $V(G_i)\subseteq \mathbb{E}$ can be similarly proved. If for some $i$, $V(G_i)=\{v_{j},v_{j+1},v_{j+2},v_{j+3}\}$, then we can replace $V(G_i)$ by $\{v_{j+2},v_{j+3}\}\cup \{e_{(j-1)j},e_{j(j+1)}\}$ or $\{v_{j},v_{j+1}\}\cup\{e_{(j+2)(j+3)},e_{(j+3)(j+4)}\}$, and find the min-TDSs $(S-\{v_{j},v_{j+1}\})\cup\{e_{(j-1)j},e_{j(j+1)}\}$ or $(S-\{v_{j+2},v_{j+3}\})\cup\{e_{(j+2)(j+3)},e_{(j+3)(j+4)}\}$, respectively. Now let $V(G_i)=\{v_{j+1},v_{j+2},v_{j+3}\}$ for some $i$. Then $V(G_{i+1})\subseteq \mathbb{E}$ by \textsf{P.1}, and $j+4=\min \{m~|~e_{m(m+1)}\in V(G_{i+1})\}$ by the minimality of $|S|$. Then we can replace $S$ by the min-TDS $(S-\{v_{j+2}\})\cup \{e_{(j+3)(j+4)}\}$ of $T(P_n)$.
\vskip 0.2 true cm
Since $S$ is minimum, P.3 holds. 
\end{proof}

\begin{figure}[ht]
\centerline{\includegraphics[width=10.2cm, height=1.7cm]{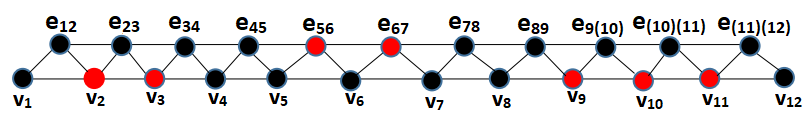}}
\vspace*{-0.3cm}
\caption{A min-TDS of $T(P_{12})$.}\label{fi:proofexample5}
\end{figure}
\begin{figure}[ht]
\centerline{\includegraphics[width=9.3cm, height=2cm]{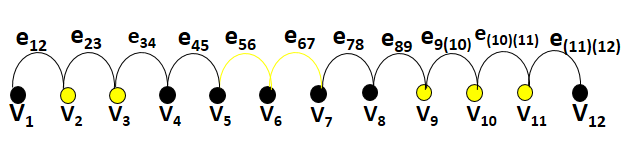}}
\vspace*{-0.3cm}
\caption{A min-TMDS of $P_{12}$.}\label{fi:proofexample5.5}
\end{figure}

\begin{prop}
\label{gamma_t T(C_n)}
For any cycle $C_n$ of order $n\geq 3$,
\begin{equation*}
\gamma_{tm}(C_n)=\left\{
\begin{array}{ll}
4\lceil n/7\rceil -3 & \mbox{if }n\equiv 1 \pmod{7},\\
4\lceil n/7\rceil -2 & \mbox{if }n\equiv 2,3 \pmod{7},\\
4\lceil n/7\rceil -1 & \mbox{if }n\equiv 4 \pmod{7},\\
4\lceil n/7\rceil  & \mbox{if }n\equiv 0,5,6 \pmod{7}.
\end{array}
\right.
\end{equation*}
\end{prop}

\begin{proof} Let $C_{n}=(V,\mathbb{E})$ be a cycle of order $n\geq 3$ in which $V=\{v_{1},v_{2},\cdots, v_{n}\}$  and $\mathbb{E}=\{e_{i(i+1)}~|~1\leq i \leq n\}$. Then $V(T(C_n))=V\cup \mathbb{E}$ and $E(T(C_n)) = \{e_{i(i+1)}v_i,e_{i(i+1)}v_{i+1}~|~ 1\leq i \leq n \}\cup \mathbb{E} \cup E(L(C_n))$ where $E(L(C_n))=\{e_{i(i+1)}e_{(i+1)(i+2)}~|~ 1\leq i \leq n\}$. In a similar way to the proof of Proposition \ref{gamma_t T(P_n)}, it can be easily verified that the sets
\begin{equation*}
\begin{array}{ll}
S_0=\{v_{7i+2}, v_{7i+3}, e_{(7i+5)(7i+6)}, e_{(7i+6)(7i+7)}~|~ 0 \leq i \leq \lfloor n/7\rfloor-1\} & \mbox{if }n \equiv 0 \pmod{7}, \\
S=S_0 \cup\{e_{(n-1)n}\}& \mbox{if }n \equiv 1 \pmod{7}, \\
S=S_0 \cup\{v_{n-1},v_{n}\} & \mbox{if }n \equiv 2,3 \pmod{7}, \\
S=S_0 \cup\{v_{n-2},v_{n-1},v_n\} & \mbox{if }n \equiv 4 \pmod{7}\\
S=S_0 \cup\{v_{n-3},v_{n-2},v_{n-1},v_n\} & \mbox{if }n \equiv 5 \pmod{7},\\
S=S_0 \cup\{v_{n-4},v_{n-3},e_{(n-2)(n-1)},e_{(n-1)n}\} & \mbox{if }n \equiv 6 \pmod{7}.
\end{array}
\end{equation*}
are min-TMDSs of $C_n$ in each case, and this completes our proof.

\vskip 0.2 true cm

The set  $\{v_2,v_3,e_{56},e_{67},v_9,v_{10},v_{11}\}$ in Figure  \ref{fi:proofexample6} shows  a min-TMDS of $C_{11}$. 

\begin{figure}[ht]
\centerline{\includegraphics[width=4cm, height=3.7cm]{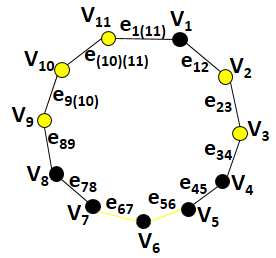}}
\vspace*{-0.3cm}
\caption{A min-TMDS of $C_{11}$.}\label{fi:proofexample6}
\end{figure}
\end{proof}
Propositions \ref{gamma_t T(P_n)} and \ref{gamma_t T(C_n)} show that the total mixed domination numbers of a cycle and a path of the same order are roughly same in the following meaning.
\begin{cor}
\label{comparing Pn, Cn}
For any integer $n\geq 3$,
\begin{equation*}
\gamma_{tm}(C_n)=\left\{
\begin{array}{ll}
\gamma_{tm}(P_n) +1 & \mbox{if }n\equiv 4,5 \pmod{7},\\
\gamma_{tm}(P_n)      & \mbox{otherwise}.
\end{array}
\right.
\end{equation*}
\end{cor}

In the next step, we calculate the total mixed domination number of a complete bipartite graph.

\begin{prop}
\label{gamma_t  (T(K_{m,n}))}
For any integers $n\geq m\geq 1$, $\gamma_{tm}(K_{m,n})=m+1$.
\end{prop}

\begin{proof}
 Let $V\cup U$ be the partition of the vertex set of the complete bipartite graph $K_{m,n}$ to the indipendent sets $V=\{v_{i}~|~ 1\leq i \leq m\}$ and $U=\{u_{j}~|~ 1\leq j \leq n\}$. Since $V\cup \{u_{1}\}$ is a TMDS of $K_{m,n}$, we have $\gamma_{tm}(K_{m,n})\leq m+1$. 

\vskip 0.2 true cm

Now, by the contrary, let $S$ be a TMDS of $K_{m,n}$ with cardinality $m$. Since the subgraph of $K_{m,n}$ induced by $V$ or 
$U$ is isomorphic to the empty graphs $\overline{K_{m}}$ or $\overline{K_{n}}$, respectively, we have $S \nsubseteq V$ and $S \nsubseteq U$. We also prove $S \nsubseteq \mathbb{E}$. For $1\leq i \leq m$ and $1\leq j \leq n$, we define $R_{i}=\{e_{ih}~|~ 
1\leq h \leq n\}$ and $C_{j}=\{e_{hj}~|~ 1\leq h \leq m\}$. Let $I=\{i~|~ 1\leq i \leq m,~R_{i} \cap S \neq \emptyset\}$ and 
$J=\{j~|~ 1\leq j \leq n,~C_{j} \cap S \neq \emptyset\}$. If $S \subseteq \mathbb{E}$, then $I\neq \{1,\cdots, m\}$ or $J\neq 
\{1,\cdots, n\}$, and so for some $i \not\in I $ or some $j \not\in J $, $v_{i}$ or $u_{j}$ is not dominated by $S$, a 
contradiction. So $S \nsubseteq \mathbb{E}$. This implies both of the sets $I_V=\{i~|~1\leq i \leq m,~v_i\in S\}$ and 
$J_U=\{j~|~1\leq j \leq n,~u_j\in S\}$ are nonempty. Because $I_V\neq \emptyset$ and $J_U=\emptyset$ imply 
$N_{T(K_{m,n})}(v_{i}) \cap S= \emptyset$ for some $i$, and $I_V=\emptyset$ and $J_U\neq \emptyset$ imply 
$N_{K_{m,n}}(u_j) \cap S= \emptyset$ for some $j$, which are contradictions. Therefore $R_i\cap S \neq \emptyset$ for each 
$i\not\in I_V$ or $C_j\cap S \neq \emptyset$ for each $j \not\in J_U$ (beacause $R_i\cap S=\emptyset$ for some $i\not\in I_V$ and $C_j\cap S=\emptyset$ for some $j\not\in J_U$ imply $N_{T(K_{m,n})}(e_{ij})\cap S=\emptyset$). Hence $|S \cap 
(\mathbb{E}\setminus \mathbb{E}_{VU})|\geq \min \{n-|I_V|, m-|J_U|\}$ in which $\mathbb{E}_{VU}=\{e_{ij}~|~ i \in I_V \mbox{ and } j \in J_U\}$, and 
so
\begin{equation*}
\begin{array}{lll}
|S| & = & m\\
     & \geq & |S \cap (\mathbb{E}\setminus \mathbb{E}_{VU})|+|I_V|+|J_U|\\
     & \geq & \min \{n-|I_V|, m-|J_U|\}+|I_V|+|J_U| \\
     & > & m,
\end{array}
\end{equation*}
a contradiction. Therefore $\gamma_{tm}(K_{m,n})=m+1$. 

\vspace{0.2cm}

The set of yellow points $\{v_1,v_2,v_{3},u_1\}$ in Figure \ref{fi:proofexample7} shows a min-TMDS of  $K_{3,3}$.
\end{proof}

\begin{figure}[ht]
\centerline{\includegraphics[width=3cm, height=2.7cm]{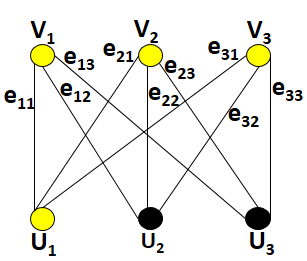}}
\vspace*{-0.3cm}
\caption{ A min-TMDS of $K_{3,3}$.}\label{fi:proofexample7}
\end{figure}


Next proposition gives the total mixed domination number of a complete graph. First a lemma.

\begin{lem}\label{K_nA}
Let $S$ be a min-TMDS of a graph $G=(V,\mathbb{E})$ in which $V=\{v_1,\cdots,v_n\}$ and $\mathbb{E}=\{e_{ij}~|~v_iv_j \mbox{ is an edge}\}$. If $A=\{v_i~|~ v_i\notin S, ~e_{ij}\in S \mbox{ for some }j\}$, then
\begin{equation*}
|S\cap \mathbb{E}|\geq
\left\{
\begin{array}{ll}
\lfloor 2|A|/3\rfloor  & \mbox{if }|A| \equiv 0 \pmod{3}, \\
\lfloor 2|A|/3\rfloor +1 & \mbox{if }|A| \not\equiv 0 \pmod{3}.
\end{array}
\right.
\end{equation*}
\end{lem}

\begin{proof}
Let $G=(V,\mathbb{E})$ be a graph in which $V=\{v_1,\cdots,v_n\}$ and $\mathbb{E}=\{e_{ij}~|~v_iv_j \mbox{ is an edge}\}$. Let $S$ be a min-TMDS of $G$ and let $A=\{v_i~|~ v_i\notin S, ~e_{ij}\in S \mbox{ for some }j\}$. For any $B\subseteq A$, we define $\mathbb{E}_{B}=\{e_{ij}\in S ~|~ \{v_i,v_j\}\cap B\neq \emptyset\}$. If $N_G(v_i)\cap A=\emptyset$ for each $v_i \in A$, then $\mathbb{E}_{A}$ and $A$ have same cardinality, and so $|S\cap \mathbb{E}|\geq |\mathbb{E}_A|=|A|\geq  \lfloor 2|A|/3\rfloor+1$, as desired. Therefore, we assume that there exist two vertices $v_i,v_j\in A$ while $e_{ij}\in S$, and continue our proof by induction on $|A|$. It can be easily verified that for any $B\subseteq A$ with cardinality at most 2, the following inequality (\ref{m-3}) holds, and we assume it holds for any set of cardinality less than $|A|$. Then we may assume $e_{i\ell}\in N_{T(G)}(e_{ij})\cap S$ for some $\ell\neq j$. By using the induction hypothesis for the set $B=A\setminus \{v_i,v_j,v_{\ell}\}$, which has cardinality $m-3$ or $m-2$, we have
\begin{equation}\label{m-3}
|\mathbb{E}_B|\geq
\left\{
\begin{array}{ll}
\lfloor 2|B|/3\rfloor  & \mbox{if }|B| \equiv 0 \pmod{3}, \\
\lfloor 2|B|/3\rfloor +1 & \mbox{if }|B| \not\equiv 0 \pmod{3}.
\end{array}
\right.
\end{equation}
Now by inequality (\ref{m-3}) and the fact that $|\mathbb{E}_A|\geq |\mathbb{E}_B\cup \{e_{ij}, e_{i\ell}\}|=|\mathbb{E}_B|+2$, our proof will be completed.
\end{proof}

\begin{prop}
\label{gamma_t T(K_n)}
For any complete graph $K_n$ of order $n\geq 2$,
\begin{equation*}
\gamma_{tm}(K_n)=\left\{
\begin{array}{ll}
\lfloor 2n/3\rfloor  & \mbox{if }n \equiv 0 \pmod{3}, \\
\lfloor 2n/3\rfloor +1 & \mbox{if }n \equiv 1,2 \pmod{3}.
\end{array}
\right.
\end{equation*}
\end{prop}

\begin{proof}
Let $ K_{n} $ be a complete graph with the vertex set $V=\{v_{1}, v_{2},\cdots, v_{n}\}$ and the edge set $\mathbb{E}$. By Propositions \ref{gamma_t T(P_n)} and \ref{gamma_t T(C_n)}, we may assume $n\geq 4$. For any arbitrary TMDS $S$ of $ K_{n}$ we show
 \begin{equation}\label{ajjjj}
|S|\geq \left\{
\begin{array}{ll}
\lfloor 2n/3\rfloor  & \mbox{if }n \equiv 0 \pmod{3}, \\
\lfloor 2n/3\rfloor +1 & \mbox{if }n \equiv 1,2 \pmod{3}.
\end{array}
\right.
\end{equation}
If $S\cap \mathbb{E} =\emptyset$, then $|S|\geq n-1 \geq \lfloor 2n/3\rfloor +1$, and there is nothing to prove (because otherwise, for any two vertices $v_{i}$ and $v_{j}$ out of $S$, the edge $e_{ij}$ can not be dominated by $S$). Also if $ S\cap V=\emptyset$, then there exists an edge $e_{ik_i}\in S$ for dominating $v_i$ by $S$, and so inequality (\ref{ajjjj}) holds, by Lemma \ref{K_nA}. Therefore we assume $ S\cap V \neq \emptyset$ and $S\cap \mathbb{E} \neq \emptyset$. Let $|S\cap V|=\ell\geq 1$. Then the set 
\[
A=\{v_i~|~ v_i\in V-S \mbox{ and } e_{ij}\in S \mbox{ for some } j \}
\]
has cardinality at least $n-\ell-1$ (because otherwise, for any two vertices $v_i, v_i\in V-S$, the edge $e_{ij}$ does not dominate by $S$), and so
\begin{equation*}
|S\cap \mathbb{E} | \geq \left\{
\begin{array}{ll}
\lfloor 2(n-\ell-1)/3\rfloor  & \mbox{if }n \equiv \ell+1 \pmod{3}, \\
\lfloor 2(n-\ell-1)/3\rfloor +1 & \mbox{if }n\equiv \ell,\ell+2 \pmod{3},
\end{array}
\right.
\end{equation*}
by Lemma \ref{K_nA}. Hence
\begin{equation*}
|S|=|S\cap V|+|S\cap \mathbb{E} |  \geq \left\{
\begin{array}{ll}
\lfloor (2n+\ell-2)/3\rfloor  & \mbox{if }n \equiv \ell+1 \pmod{3},\\
\lfloor (2n+\ell-2)/3\rfloor +1 & \mbox{if }n \equiv \ell,\ell+2 \pmod{3},
\end{array}
\right.
\end{equation*}
which implies
\begin{equation*}
|S|\geq \left\{
\begin{array}{lll}
\lfloor 2n/3\rfloor  & \mbox{if }n \equiv 0 \pmod{3}, \\
\lfloor 2n/3\rfloor +1 & \mbox{if }n\equiv 1\pmod{3},\\
\lfloor 2n/3\rfloor +1 & \mbox{if }n \equiv 2 \pmod{3} \mbox{ and } \ell\neq 1.
\end{array}
\right.
\end{equation*}

Now we discuse on the only remained case $n\equiv 2 \pmod{3}$ and $\ell=1$. Let $S\cap V=\{v_n\}$ and $\mathbb{E}_A=\{e_{ij}\in S~|~ \ \{v_i,v_j\}\cap A\neq \emptyset\}$. Then $n-2\leq |A|\leq n-1$, and
\begin{equation*}
|\mathbb{E}_A| \geq \left\{
\begin{array}{lll}
\lfloor 2(n-2)/3\rfloor  & \mbox{if }|A|=n-2 \equiv 0 \pmod{3}, \\
\lfloor 2(n-1)/3\rfloor +1 & \mbox{if }|A|=n-1 \equiv 1 \pmod{3}.
\end{array}
\right.
\end{equation*}
Since $|A|=n-1$ implies $|S| = |S\cap V|+|S\cap \mathbb{E}| \geq 1+|\mathbb{E}_A| \geq \lfloor 2n/3\rfloor +1$, as desired, we assume $|A|=n-2$. For some $p\neq n$, let $e_{pn}\in S$ be an edge that dominates $v_n$. Then $v_p\in A$ and $e_{pn}\in \mathbb{E}_A$, and so $|A\setminus \{v_p\}|=n-3 \equiv 2 \pmod{3}$. Hence $|\mathbb{E}_{A\setminus \{v_p\}}|\geq \lfloor 2(n-3)/3\rfloor +1$ by (\ref{m-3}).
Now the facts $e_{pn}\notin \mathbb{E}_{A\setminus \{v_p\}}$ and $\mathbb{E}_{A\setminus \{v_p\}}\cup\{e_{pn},v_n\}\subseteq S$ imply $|S|\geq |\mathbb{E}_{A\setminus \{v_p\}}|+2\geq \lfloor 2n/3\rfloor +1$, as desired. On the other hand, since each of the sets
\begin{equation*}
\begin{array}{ll}
S_0=\{e_{(3i+1)(3i+2)}, e_{(3i+2)(3i+3)}~|~ 0 \leq i \leq k-1\}& \mbox{if }n \equiv 0 \pmod{3}, \\
S_1=S_0\cup\{e_{(3k)(3k+1)}\} & \mbox{if }n \equiv 1 \pmod{3}, \\
S_2=S_0\cup\{e_{(3k)(3k+1)}\}, e_{(3k+1)(3k+2)}\} & \mbox{if }n \equiv 2 \pmod{3},
\end{array}
\end{equation*}
is a TMDS of $K_n$ with the minimum cardinality when $k=\lfloor n/3 \rfloor$, we have proved
\begin{equation*}
\gamma_{tm}(K_n)=\left\{
\begin{array}{ll}
\lfloor 2n/3\rfloor  & \mbox{if }n \equiv 0 \pmod{3}, \\
\lfloor 2n/3\rfloor +1 & \mbox{if }n \equiv 1,2 \pmod{3}.
\end{array}
\right.
\end{equation*}
 The set of red edges $\{e_{12},e_{23},e_{34}\}$ in Figure \ref{fi:proofexample8} shows a min-TMDS of $K_{4}$.
\begin{figure}[ht]
\centerline{\includegraphics[width=4.1cm, height=2.7cm]{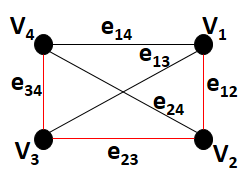}}
\vspace*{-0.3cm}
\caption{A min-TMDS of $K_{4}$.}\label{fi:proofexample8}
\end{figure}
\end{proof}
Before giving the total mixed domination number of a wheel, as we promise in the proof of Theorem \ref{max{gama_t  G, gama_t  L(G)} =< gama_t T(G)=<gama_t  G+gama_t  L(G)}, we show that  the lower bound in Theorem \ref{max{gama_t  G, gama_t  L(G)} =< gama_t T(G)=<gama_t  G+gama_t  L(G)} is tight by Proposition \ref{gamma_t T(K_n)} and calculating the total domination number of the line graph of a complete graph in the next proposition.

\begin{prop}
\label{gama_t(L(K_n))=2n/3}
For any complete graph $K_{n}$ of order $n\geq 4$, $\gamma_t(L(K_{n}))=\lfloor 2n/3 \rfloor$.
\end{prop}
\begin{proof}
Let $K_{n}=(V,\mathbb{E})$ be a complete graph of order $n\geq 4$ with vertex set $V=\{v_{1}, \cdots, v_{n}\}$ and the edge set $\mathbb{E}$. Then $V(L(K_{n}))=\mathbb{E}$. For any TDS $S$ of $L(K_{n})$, let $I$ be the set of all indices of the vertices of $S$. Obviousely for any three indices $1\leq i<j<k\leq n$, $|\{i,j,k\} \cap I| \geq 2$ because if $i,j\not \in I$, for example, then the vertex $e_{ij}$ can not be dominated by $S$. Thus $|S| \geq \lfloor 2n/3 \rfloor $, and since the sets
\begin{equation*}
\begin{array}{ll}
S_{0}=S_{2} \cup \{e_{(n-2)(n-1)} \} & \mbox{if }n \equiv 0 \pmod{3},\\
S_{1}=\{e_{(3i+1)(3i+2)}, e_{(3i+2)(3i+3)},~|~ 0 \leq i \leq \lfloor n/3 \rfloor -1\} & \mbox{if }n \equiv 1 \pmod{3}, \\
S_{2}=S_{1} \cup \{e_{(n-2)(n-1)} \} & \mbox{if }n \equiv 2 \pmod{3},
\end{array}
\end{equation*} 
are TDSs of $L(K_{n})$ in each of the cases, the result holds.
\end{proof}

\begin{prop}
\label{gamma_t(T(W_n))=lceil n/2 rceil +1}
 For any wheel $W_n$ of order $n+1\geq 4$, $\gamma_{tm}(W_n)=\lceil n/2 \rceil +1$.
\end{prop}

\begin{proof}
 Let $W_n=(V,\mathbb{E})$ be a wheel of order $n+1\geq 4$ with the vertex set $ V=\{v_{i}~|~ 0 \leq i \leq n \} $ and the edge set $\mathbb{E}=\{e_{0i},e_{i(i+1)}~|~1\leq i \leq n\}$. Since $S=\{v_{0}\}\cup \{v_{2i-1}~|~1\leq i \leq \lceil n/2\rceil\}$ is a TMDS of $W_n$, we have $\gamma_{tm}(W_n)\leq \lceil n/2 \rceil +1$. 

\vspace{0.2cm}

In the sequel, we show $\gamma_{tm}(W_n)\geq \lceil n/2 \rceil +1$. Let $S$ be an arbitrary TMDS of $W_{n}$. If $S\cap {\mathbb{E}}=\emptyset$, then since $N(e_{i(i+1)})\cap S \neq \emptyset$ for each $1\leq i\leq n$, $S\cap \{v_i,v_{i+1}\}\neq \emptyset$, and so $|S|\geq \lceil n/2\rceil$. Since  we have nothing to prove when $\{v_1,\cdots,v_n\}\subseteq S$, we assume $v_i\notin S$ for some $1\leq i\leq n$. This implies $v_0\in S$ (because $S$ dominates $e_{0i}$ and $S\cap \mathbb{E}=\emptyset$), and so $|S|\geq \lceil n/2\rceil+1$. Now let $S\cap V=\emptyset$. Then, for dominating every vertex $v_i\in V$ by $S$, there exists an edge $e_{pi}\in S$ for some $p\neq i$. By knowing $N_{W_n}(e_{pi})=\{v_p,v_i\}$, we conclude $S$ has cardinality at least $\lceil (n+1)/2\rceil$ which is $\lceil n/2\rceil+1$ for even $n$ and is $\lceil n/2\rceil$ for odd $n$. Since the subgraph of $W_n$ induced  by $S$ is connected and $S\subseteq \mathbb{E}$, we obtain $|\{v_i~|~e_{ij}\in S \mbox{ for some }j\}|<n+1$ if $|S|= \lceil (n+1)/2\rceil$ and $n$ is odd, a contradiction. Thus $|S|\geq \lceil (n+1)/2\rceil+1=\lceil n/2\rceil+1$ for odd $n$.

\vspace{0.2cm}

Thus $S\cap V\neq\emptyset$ and $S\cap \mathbb{E}\neq\emptyset$. By assumption $|S\cap V|=\ell$ it is sufficient to prove $|S\cap \mathbb{E}|\geq \lceil n/2\rceil-\ell+1$. Let $\mathbb{E}_{0}=\{e_{i(i+1)}~|~|\{v_{i},v_{i+1}\}\cap S|=0 \mbox{ for }1\leq i\leq n\}$. Since every $e_{i(i+1)} \in \mathbb{E}_{0}$ must be dominated by an edge $e_{pq}\in S$ in which $|\{p,q\}\cap \{i,i+1\}|=1$, the set $\mathbb{E}_{00}=\{e_{pq}\in S~|~e_{pq} \mbox{ dominates an edge } e_{i(i+1)} \in \mathbb{E}_{0} \}$ is not empty, and more $|\mathbb{E}_{00}|\geq \lceil |\mathbb{E}_{0}|/2\rceil $ because every $e_{pq}\in \mathbb{E}_{00}$ is adjacent to at most two edges in $\mathbb{E}_{0}$. Let $\mathbb{E}_{1}=\{e_{i(i+1)}~|~|\{v_{i},v_{i+1}\}\cap S|\neq 0 \mbox{ for }1\leq i\leq n\}$. If $v_{0}\in S$, then $|\mathbb{E}_1|\leq 2(\ell-1)$, and so $|\mathbb{E}_0|\geq n-2\ell+2$ which implies $|S\cap \mathbb{E}|\geq \lceil |\mathbb{E}_0|/2 \rceil = \lceil (n-2\ell+2)/2\rceil=\lceil n/2\rceil-\ell+1$, as desired. 

\vspace{0.2cm}

Therefore we may assume $v_{0}\notin S$. Then $|\mathbb{E}_1|\leq 2\ell$ and so $|\mathbb{E}_0|\geq n-2\ell$. Let $|\mathbb{E}_0|\leq n-2\ell+1$ by the contrary. So $2\ell-1 \leq |\mathbb{E}_1|\leq 2\ell$. Since by the assumption $|\mathbb{E}_1|=2\ell$ we reach to this contradiction that the subgraph of $W_n$ induced by $S\cap V$ contains $\ell$ isolate vertices, we may assume $|\mathbb{E}_1|=2\ell-1$. Again, since the subgraph of $W_n$ induced  by $S\cap V$ does not have isolate vertex, we must have $\ell=2$, and so $|\mathbb{E}_0|=n-2\ell+1=n-3$. Since obviousely $|S\cap \mathbb{E}| \geq \lceil |\mathbb{E}_0|/2 \rceil = \lceil n/2\rceil-\ell+1$ for even $n$, let $n$ be odd. Without loss of generality, we assume $S\cap V=\{v_1,v_2\}$, and so $\mathbb{E}_0=\{e_{i(i+1)}~|~ 3\leq i\leq n-1\}$. \\

By the contrary let $|S\cap \mathbb{E}| =(n-3)/2$. Since there is nothing to prove for $n=3$ by Proposition \ref{gamma_t T(K_n)}, we assume $n\geq 5$. For $n=5$, since $v_{4}$ does not dominated by $S-\mathbb{E}$, $|S\cap \mathbb{E}|=1$ implies $S\cap \mathbb{E}=\{e_{i4}\}$ for some $i\neq 4$, that is, $S=\{v_1,v_2,e_{i4}\}$. But since $e_{i4}$ does not dominated by $S$, we reach contradiction. Thus $|S\cap \mathbb{E}| >(n-3)/2=1$ and so $\gamma_{tm}(W_5)=\lceil 5/2 \rceil +1=4$. Therefore, in the sequel, we assume $n\geq 7$. If $S\cap \mathbb{E} \subseteq \mathbb{E}_0$, then $S\cap \mathbb{E}=\mathbb{E}_{0}-\{e_{(2i)(2i+1)}~|~ 2\leq i \leq (n-3)/2+1 \}$ or $S\cap \mathbb{E}=\{e_{2i(2i+1)}~|~ 2\leq i \leq  (n-3)/2 \} \cup \{\alpha\}$ where $\alpha\in \{e_{(n-2)(n-1)},e_{(n-1)n} \}$, which imply one of the edges $e_{0n}$ or $e_{03}$ does not respectively dominated by $S$, a contradiction. Thus $|S \cap \{e_{0i}~|~1\leq i \leq n\}|=m\geq 1$. Let also $V^{'}=V-N_G(\{v_1,v_2\})=\{v_{4},\ldots, v_{n-1}\}$, $\mathbb{E}^{'}_0=\{e_{i(i+1)} \in S\cap \mathbb{E}_0~|~ \{e_{0i}, e_{0(i+1)}\} \cap S \neq \emptyset\}$ and $\mathbb{E}^{''}_0=\mathbb{E}_0 - \mathbb{E}^{'}_0$. So 
\begin{equation}
\label{|S cap T_0|=(n-3)/2 - m}
|S\cap \mathbb{E}_0|=(n-3)/2 - m=|\mathbb{E}^{'}_0|+|\mathbb{E}''_0|.
\end{equation}
Since $N_{W_n}(e_{i(i+1)})=\{v_i,v_{i+1}\}$ for $e_{i(i+1)} \in \mathbb{E}'_0$ and $|N(e_{0j})\cap \{v_i,v_{i+1}\}|=1$ when $e_{0j}\in \{e_{0i},e_{0(i+1)}\} \subset S$, we have $ |N_{T(W_n)}(S\cap \mathbb{E}^{'}_0) \cap V^{'}|\leq m+|\mathbb{E}'_0|$. Since the subgraph $W_n[\mathbb{E}''_0]$ of $W_n$ induced by $\mathbb{E}''_0$ dominates the most number of vertices in $V'$ when it has as possible as the most number of the complete graphs $K_2$ as induced subgraphs, we conclude that at least two edges of $\mathbb{E}''_0$ are needed for dominating every three vertices of $V'$ by $\mathbb{E}''_0$, and so $|N_{T(W_n)}(S\cap \mathbb{E}''_0) \cap V^{'}| \leq 3|\mathbb{E}''_0|/2$. Hence 
\begin{equation}
\label{N(S cap T_0) cap V'| =< m+|T'_0|+3|T''_0|/2}
|N_{T(W_n)}(S\cap \mathbb{E}_0) \cap V^{'}| \leq m+|\mathbb{E}'_0|+3|\mathbb{E}''_0|/2.
\end{equation}
Now by knowing $V^{'} \subseteq N_{T(W_n)}(S\cap \mathbb{E}_0)$ which implies $|N_{T(W_n)}(S\cap \mathbb{E}_0)|\geq |V^{'}|=n-4$, and relations (\ref{|S cap T_0|=(n-3)/2 - m}) and (\ref{N(S cap T_0) cap V'| =< m+|T'_0|+3|T''_0|/2}), we obtain $|\mathbb{E}^{''}_0|\geq n-5$, and so $m= 1$.  Then $|S\cap \mathbb{E}_0|=(n-3)/2-1=(n-5)/2$. Thus the number of vertices of $V'$ dominated by $S\cap \mathbb{E}_0$ is at most $3(n-5)/4+1$ which is less than $n-4=|V'|$ when $n\geq 7$. Therefore  $|S\cap \mathbb{E}| \geq (n-3)/2+1=(n-1)/2=\lceil n/2 \rceil-\ell+1$, as desired. 

\vskip 0.1 true cm

The set of yellow points $\{v_0,v_1,v_3,v_5\}$ in Figure  \ref{fi:proofexample9} shows a min-TMDS of $W_5$.
\begin{figure}[ht]
\centerline{\includegraphics[width=4.5cm, height=3.5cm]{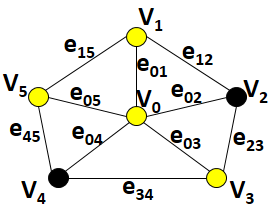}}
\vspace*{-0.3cm}
\caption{A min-TMDS of $W_5$.}\label{fi:proofexample9}
\end{figure}
\end{proof}

We know $\gamma_{tm}(G)> \gamma_t(G)$ for almost all graphs. As we saw in some graphs such as complete graphs and wheels, $ \gamma_{tm}(G)-\gamma_t(G)\rightarrow \infty$ when $n\rightarrow \infty$ for many graphs $G$. So, we end our paper with the following important problem. 

\begin{prob}
\label{gama_t  G =< gama_t T(G)}
Find some real number $\alpha > 1$ such that for any graph $G$, $\gamma_{tm}(G) \geq \alpha \gamma_t(G)$.
\end{prob}


\end{document}